\documentclass[runningheads]{llncs}
\usepackage[T1]{fontenc}
\usepackage{graphicx} 
\usepackage{algorithm}
\usepackage{algpseudocode}
\usepackage[colorlinks=true]{hyperref}
\usepackage{tikz,amsmath}
\usepackage{color}
\usetikzlibrary{decorations.pathreplacing}
\usetikzlibrary{positioning, calc,chains,fit,shapes}
\usepackage{enumerate}
\usepackage{booktabs,array}
\usepackage{cleveref}
\usepackage{tikz}
\usepackage{comment}

\newcommand{\EGC}{Erd\H{o}s-Gy\'{a}rf\'{a}s conjecture}
\newcommand{\EG}{Erd\H{o}s and Gy\'{a}rf\'{a}s}
\newcommand{\E}{Erd\H{o}s}
\newcommand{\TRUE}{\texttt{True}}
\newcommand{\FALSE}{\texttt{False}}

\begin{document}
\title{Erd\H{o}s-Gy\'{a}rf\'{a}s conjecture on graphs without long induced paths\thanks{A major part of this work was done when the first and the last authors were at IIT Dharwad. This project is supported by ANRF (SERB) MATRICS grant MTR/2022/000692.}}
%
%
\author{Anand Shripad Hegde\inst{1,2} \and
R. B. Sandeep\inst{1} \and
P. Shashank\inst{1,2}}
\authorrunning{A. Hegde et al.}
%
\institute{Department of Computer Science and Engineering, Indian Institute of Technology Dharwad, Dharwad, India \\
\email{\{cs20bt058.alum24, sandeeprb, cs20bt048.alum24\}@iitdh.ac.in} \\ \and
Arista Networks, Bengaluru, India}
\maketitle              
\begin{abstract}
Erd\H{o}s and Gy\'{a}rf\'{a}s conjectured in 1994 that every graph with minimum degree at least $3$ has a cycle of length a power of $2$. In 2022, Gao and Shan (Graphs and Combinatorics) proved that the conjecture is true for $P_8$-free graphs, i.e., graphs without any induced copies of a path on $8$ vertices. In 2024, Hu and Shen (Discrete Mathematics) 
improved this result by proving that the conjecture is true for $P_{10}$-free graphs. With the aid of a computer search, we improve this further by proving that the conjecture is true for $P_{13}$-free graphs.

\end{abstract}
In 1994, \EG~\cite{DBLP:journals/dm/Erdos97b}\footnote{In a paper published in 1997 \cite{DBLP:journals/dm/Erdos97b}, \E\ wrote: ``About three years ago, Gy\'{a}rf\'{a}s and I thought at first that if $G$ is any graph every vertex of which has degree $\geq 3$ then our $G$ has a cycle of length $2^k$ for some $k$. We are convinced now that this is false and no doubt there are graphs for every $r$ every vertex of which has degree $\geq r$ and which contains no cycle of length $2^k$, but we never found a counterexample even for $r=3$.''}
conjectured that  every graph with minimum degree at least $3$ has a cycle of length a power of $2$. The conjecture has been verified for the following classes of graphs - $3$-connected cubic planar graphs~\cite{DBLP:journals/combinatorics/HeckmanK13}, claw-free planar graphs~\cite{daniel2001result}, $K_{1,m}$-free graphs with some additional degree constraints~\cite{shauger1998results}, and some families of Cayley graphs~\cite{ghaffari2018erdHos}. 
Liu and Montgomery~\cite{liu2023solution} proved that there exists a large constant such that every graph with average degree at least the constant has a cycle with length a power of 2. 
This disproved \E ' later conviction that the conjecture is false for every minimum degree at least 3~\cite{DBLP:journals/dm/Erdos97b}. 
Extensive computer searches have been done to show that a counterexample has at least $17$ vertices, a cubic counterexample has at least $30$ vertices~\cite{markstrom2002extremal}, and a bipartite counterexample has at least $30$ vertices~\cite{nowbandegani2011experimental}.

In 2022, Gao and Shan~\cite{DBLP:journals/gc/GaoS22} proved that the conjecture is true for $P_8$-free graphs, i.e., graphs without any induced copies of a path on $8$ vertices. Very recently, Hu and Shen~\cite{DBLP:journals/dm/HuS24} proved that the conjecture is true for $P_{10}$-free graphs. Inspired by a case analysis in \cite{DBLP:journals/gc/GaoS22}, we prove with the aid of a computer search that the conjecture is true for $P_{13}$-free graphs. 
We prove a stronger statement for $P_{12}$-free graphs that every $P_{12}$-free graphs with minimum degree at least 3 has a 4-cycle or an 8-cycle. This improves upon the similar result obtained for $P_{10}$-free graphs~\cite{DBLP:journals/dm/HuS24}.

The key backtracking algorithm that we implemented is shown in Figure~\ref{algo}. The implementation is available at ~\cite{Hegde_Verifier_for_ErdosGyarfas_2024}. The algorithm \textsc{explore} takes two inputs: a graph $G$ and an integer $k\geq 3$. The vertices of $G$ are assumed to be $v_0, v_1, \ldots, v_{n-1}$. It outputs \texttt{False}, 
if $G$ can be \textit{extended} to a $P_k$-free counterexample, i.e., if there exists a $P_k$-free counterexample $G^*$ to \EGC\ with a vertex labeling $v_0,v_1,\ldots, v_{n^*-1}$ (where $n^*\geq n$) such that $\{v_i,v_j\}$, for $0\leq i<j\leq n-2$, is an edge in $G$ if and only if it is an edge in $G^*$, and the set of neighbors of $v_{n-1}$ in $G$ is a subset of 
its neighbors in $G^*$. The algorithm returns \texttt{True} otherwise. We say that a subset $S$ of $\{v_0,v_1,\ldots,v_{n-2}\}$ is a set of potential safe neighbors of $v_{n-1}$ if $\{v_i, v_{n-1}\}$ is not an edge (for $v_i\in S$) in $G$ and making all vertices in $S$ neighbors of $v_{n-1}$ 
does not create any forbidden cycles (a cycle is forbidden if it is of length a power of $2$). For each such set $S$ of safe potential neighbors of $v_{n-1}$, we add the corresponding edges, and then check whether the new graph is a counterexample or not. If a $P_k$-free counterexample is found, we return \texttt{False}. If the updated graph $G$ is 
not a counterexample and it is $P_k$-free, then we find a vertex, anchor\_node with degree less than $3$, and make it adjacent to a new vertex $v_n$. We call \textsc{explore} with this updated graph $G$ and $k$ and return the result if it is \texttt{False}. If the recursive calls return \texttt{True} for each set of potential safe neighbors, then 
return \texttt{True} at the end. We run \textsc{explore} with $P_k, k$ for $k\geq 3$. 
If we obtain that the function returns \textsc{True} for every $3\leq k\leq t$, then we claim that the conjecture is true for $P_{t}$-free graphs.

The function \textsc{get\_largest\_low\_degree\_vertex} returns the vertex with largest index having a degree less than $3$. It returns $-1$ if the minimum degree is at least $3$. By $G[A]$, where $A$ is a subset of vertices of $G$, we denote the graph induced by $A$ in $G$. If $A = \{u_1,u_2,\ldots,u_p\}$, then $G[u_1,u_2,\ldots,u_p]$ denotes the graph $G[A]$. 
A graph is a minimal counterexample, if none of its proper induced subgraphs are counterexamples.

\begin{figure}
    \caption{The algorithm}
    \label{algo}
    \begin{algorithmic}[1]
    \Function{explore}{$G, k$}
    \Comment{$V(G) = \{v_0, v_1,\ldots,v_{n-1}\}$}
      \For{each set $S$ of potential safe neighbors of $v_{n-1}$}
        \State Add edge $\{v_i, v_{n-1}\}$ to $G$, for each $v_i\in S$.
        \If{there is an induced $P_k$ in $G$}
          \State continue
        \EndIf
        \State anchor\_node $\gets$ \Call{get\_largest\_low\_degree\_vertex}{$G$}
        \If{anchor\_node $==$ -1}
          \State return \FALSE \Comment{G is a counterexample.}
        \EndIf
        \State Add a node $v_n$ and the edge $\{v_n, anchor\_node\}$ to $G$
        \If{not \Call{explore}{$G,k$}}
          \State return \FALSE
        \EndIf
        \State Remove node $v_n$ from $G$.
        \State Remove edge $\{v_i, v_{n-1}\}$ from $G$, for each $v_i\in S$.
      \EndFor
      \State return \TRUE
    \EndFunction
    \end{algorithmic}
\end{figure}

\begin{lemma}
\label{lem}
    Let $G^*$ be a minimal counterexample to \EGC, such that $G^*$ is $P_k$-free for an integer $k\geq 3$. Let $\{v_0,v_1,\ldots,v_{n^*-1}\}$ be the set of vertices of $G^*$. Let $G$ be a graph with the vertex set $\{v_0,v_1,\ldots,v_{n-1}\}$, where $3\leq n\leq n^*$, such that the following conditions are satisfied.
    \begin{enumerate}[i.]
        \item The pair $\{v_i,v_j\}$, for $0\leq i < j \leq n-2$,
        is an edge in $G^*$ if and only if it an edge in $G$.
        \item If $v_i$, for $0\leq i\leq n-2$, is a neighbor of $v_{n-1}$ in $G$, then so is in $G^*$. 
    \end{enumerate}
    Then \textsc{explore($G,k$)} returns \texttt{False}.
\end{lemma}
\begin{proof}
We prove by induction on $n^*-n$. For the base case, assume that $n^*-n=0$. Let $S$ be the set of neighbors in $G^*$ which are not neighbors of $v_{n-1}$ in $G$. This set $S$ is discovered in one iteration of the \texttt{for} loop (line 2). Then $G$ is updated (line 3) with the corresponding edges and the updated graph becomes isomorphic to $G^*$. Since $G^*$ has no induced $P_k$ and has minimum degree at least 3, the \texttt{if} condition in line 4 fails and that in line 8 succeeds and the algorithm returns \texttt{False}. 

Now, assume that $n^*-n > 0$. As before, let $S$ be the set of neighbors among $\{v_0,v_1,\ldots,v_{n-2}\}$ of $v_{n-1}$ in $G^*$ which are not neighbors of $v_{n-1}$ in $G$. This set $S$ is discovered in one iteration of the \texttt{for} loop and the graph $G$ is updated with the corresponding edges. Therefore, $\{v_i,v_j\}$, for $0\leq i < j \leq n-1$,
is an edge in $G^*$ if and only if it an edge in $G$. Since $G^*$ is $P_k$-free and $G$ is an induced subgraph of $G^*$,
the \texttt{if} condition in line 4 fails. Since $G^*$ is a minimal counterexample, $G$ must have a vertex with degree less than 3. Let $i$ be the largest index such that $v_i$ is a vertex of degree at most 2 in $G^*[v_0,v_1,\ldots,v_{n-1}]$. Then $v_i$ is returned as the anchor\_node in line 7. Clearly, 
$v_i$ must have a neighbor $v_j$ in $G^*$ such that $v_j\in \{v_n,v_{n+1},\ldots,v_{n^*-1}\}$.
Without loss of generality, assume that $j = n$. Now, the vertex $v_{n}$ and the edge $\{v_{n}, \text{anchor\_node}\}$ are added to $G$ (line 11). Hence the number of vertices in $G$ got incremented by one. Further, $G$ trivially satisfies both the conditions in the statement of the lemma. Now the proof follows from the induction hypothesis.
\end{proof}

Corollary~\ref{cor} follows from Lemma~\ref{lem}.

\begin{corollary}
\label{cor}
Let $G^*$ be a minimal counterexample to \EGC\ and let $k$ be the smallest integer such that $G^*$ is $P_k$-free but has an
induced $P_{k-1}$. Let $G$ be the path $v_0v_1\ldots v_{k-1}$. 
Then \textsc{explore}($G, k$) returns \texttt{False}. 
\end{corollary}

\begin{theorem}
\label{thm}
Every $P_{13}$-free graph with minimum degree at least $3$ has a cycle of length a power of $2$.
\end{theorem}
\begin{proof}
    We implemented \textsc{explore} and ran it with $G, k$, where 
    $G$ is the path $v_0v_1\ldots v_{k-1}$, for $k$ from $3$ to $13$.
    In all the executions, the program returned \texttt{True}. Now, the statement follows from Corollary~\ref{cor} and the fact that $P_2$-free graphs have no edges.
\end{proof}

We have a different implementation of the algorithm in which only cycles of lengths 4 and 8 are forbidden (see the branch `4-8-cycles' in the repository~\cite{Hegde_Verifier_for_ErdosGyarfas_2024}). Using this, we obtain a stronger result for $P_{12}$-free graphs.
\begin{theorem}
    \label{thm:4-8}
    Every $P_{12}$-free graph with minium degree at least $3$ has a 4-cycle or an 8-cycle.
\end{theorem}

This improves upon the result by Hu and Shen~\cite{DBLP:journals/dm/HuS24} that every $P_{10}$-free graph with minimum degree at least 3 has a 4-cycle or an 8-cycle.

We implemented \cite{Hegde_Verifier_for_ErdosGyarfas_2024} the algorithm\footnote{There are slight differences between the algorithm in Figure~\ref{algo} and the implemented algorithm. The differences are trivial and have no implications on the correctness. For example, the implemented algorithm prints a counterexample, if exists, and exits, instead of returning \FALSE. In the implemented algorithm, the arguments passed to \textsc{explore} are $G$ and an anchor\_node, instead of $G$ and $k$. This is because, $k$ is available as part of the graph object and passing anchor\_node helps in efficient execution.} 
in C++. We also have a faster but memory-intensive parallel implementation using Cilk~\cite{blumofe1995cilk} (see `cilk' branch of the repository). The time taken by these implementations when ran on a server ($2.6$ GHz CPU, 72 cores) is shown in Table~\ref{tab:time}. The parallel execution was terminated due to out of memory error when ran for $k = 14$. A serial execution of the modified implementation used for Theorem~\ref{thm:4-8} took 1 minute 39 seconds for $k=11$, and 3 hours 9 minutes for $k=12$. We conducted extensive tests to obtain evidences for the correctness of the implementation. The details are given in the appendix.

\setlength\extrarowheight{2pt}
\begin{table}[h]
    \centering
\begin{tabular}{ >{\hfill}p{3.25cm}|p{1.05cm}p{.85cm}p{.85cm}p{1.5cm}p{1.5cm}}
$k$ & $\leq 9$ & 10 & 11 & 12 & 13\\
\hline
Time taken (C++) & $<$0.2s & 2s & 32s & 31m 32s & 11h 56m \\
Time taken (Cilk) & $<$0.05s & 0.1s & 1.4s& 29s & 17m 17s \\
\end{tabular}
\vspace*{5mm}
    \caption{Time taken by an implementation of Algorithm~\ref{algo} for various values of $k$. Seconds, minutes, and hours are represented by `s', `m', and `h' respectively.}
    \label{tab:time}
\end{table}

The technique that we used is not helpful for proving the conjecture for $H$-free graphs when $H$ is not a path. Assume that $H$ has a cycle. Since an infinite tree with minimum degree at least 3 neither has $H$ as an induced subgraph nor has a cycle of length a power of 2 as a subgraph, the algorithm will run for ever. 
If $H$ is a tree but not a path, then we can come up with an $H$-free infinite graph by replacing each vertex of an infinite tree (where every vertex has degree 3) with a clique of 3 vertices and making each such vertex adjacent to exactly one vertex in the clique formed for a neighbor. Since this graph is claw-free, every induced tree in it is a path. We end with the following question: Is Algorithm~\ref{algo} capable of resolving the conjecture for $P_{k}$-free graphs for every integer $k\geq 14$? In other words, is there an infinite graph with minimum degree at least 3 neither having an induced $P_{k}$ nor having a cycle of length a power of 2?

\vspace{\baselineskip}
\noindent
\textbf{Acknowledgement.} We thank Nikhil Hegde for suggesting ways to boost the performance of our code. We thank Paweł Rzążewski for helpful comments on an initial draft of this paper.
\bibliographystyle{plain}
\bibliography{main}

\newpage
\appendix
\section*{Appendix: Correctness testing}
The C++ program implementing the algorithm can be found at \cite{Hegde_Verifier_for_ErdosGyarfas_2024}. 
We conducted various tests to verify the correctness of the implementation. They are listed below. The outcome of the tests were as expected and we believe that it provides strong evidences for the correctness of the implementation.

\begin{enumerate}
    \item The output of our program is in accordance with the result obtained in ~\cite{DBLP:journals/dm/HuS24} for $P_{10}$-free graphs. 
    \item By introducing logs and visualization of intermediate graphs, we manually verified the execution of the program for $3\leq k\leq 5$. The output was as expected. More details can be found in the `logs' branch of the repository.
    \item We conducted unit tests for major functions in the program. For each such functions, we tested with inputs where the corresponding outputs are known to us. More details can be found in the `tests' branch of the repository.
    \item We obtained all the four cubic graphs with minimum number (24) of vertices (found by Markström~~\cite{markstrom2002extremal})  having no 4-cycle and no 8-cycle but having a 16-cycle by guiding the execution with proper values of $k$ and by restricting the degree to be 3. Only one of them, known as Markström graph (see Figure~\ref{fig}), is planar. It is $P_{18}$-free but has an induced $P_{17}$. More details can be found in the `special-graphs' branch of the repository. 
\end{enumerate}
\begin{figure}[b]
\centering
\scalebox{0.9}{
\begin{tikzpicture}[myv/.style={circle, draw, inner sep=1.5pt, fill=black}]


  \node[myv] (a) at (50:1) {};
  \node[myv] (b) at (110:1) {};
  \node[myv] (c) at (170:1) {};
  \node[myv] (d) at (230:1) {};
  \node[myv] (e) at (290:1) {};
  \node[myv] (f) at (350:1) {};

  \node[myv] (s) at (110:2.5) {};
  \node[myv] (o) at (150:2.5) {};
  \node[myv] (p) at (190:2.5) {};
  \node[myv] (t) at (230:2.5) {};
  \node[myv] (q) at (270:2.5) {};
  \node[myv] (r) at (310:2.5) {};
  \node[myv] (u) at (350:2.5) {};
  \node[myv] (m) at (30:2.5) {};
  \node[myv] (n) at (70:2.5) {};

  \node[myv] (g) at (90:1.75) {};
  \node[myv] (h) at (130:1.75) {};
  \node[myv] (w) at (170:1.75) {};
  \node[myv] (i) at (210:1.75) {};
  \node[myv] (j) at (250:1.75) {};
  \node[myv] (x) at (290:1.75) {};
  \node[myv] (k) at (330:1.75) {};
  \node[myv] (l) at (10:1.75) {};
  \node[myv] (v) at (50:1.75) {};

  \draw[color=purple,very thick] (a) -- (b);
  \draw[color=purple,very thick] (c) -- (d);
  \draw (e) -- (f);
  
  \draw (a) -- (c);
  \draw[color=purple,very thick] (a) -- (e);
  \draw[color=purple,very thick] (c) -- (e);
  
  \draw (s) -- (o);
  \draw (o) -- (p);
  \draw (p) -- (t);
  \draw[color=purple,very thick] (t) -- (q);
  \draw[color=purple,very thick] (q) -- (r);
  \draw[color=purple,very thick] (r) -- (u);
  \draw[color=purple,very thick] (u) -- (m);
  \draw[color=purple,very thick] (m) -- (n);
  \draw[color=purple,very thick] (n) -- (s);

  \draw (g) -- (h);
  \draw (i) -- (j);
  \draw (k) -- (l);

  \draw (b) -- (h);
  \draw[color=purple,very thick] (b) -- (g);
  \draw[color=purple,very thick] (g) -- (h);

  \draw[color=purple,very thick] (d) -- (i);
  \draw (d) -- (j);
  \draw[color=purple,very thick] (i) -- (j);

  \draw (f) -- (k);
  \draw (f) -- (l);
  \draw (k) -- (l);

  \draw (v) -- (g);
  \draw (v) -- (n);
  \draw (v) -- (m);

  \draw (w) -- (i);
  \draw (w) -- (o);
  \draw (w) -- (p);

  \draw (x) -- (k);
  \draw (x) -- (q);
  \draw (x) -- (r);

  \draw[color=purple,very thick] (h) -- (s);
  \draw[color=purple,very thick] (j) -- (t);
  \draw (l) -- (u);
  
\end{tikzpicture}}
    \caption{Markström graph: the unique smallest cubic planar graph having no 4-cycle and no 8-cycle, but having a 16-cycle. The bold (purple) edges show a 16-cycle.}
    \label{fig}
\end{figure}
\end{document}